\documentclass{amsart}

\usepackage{mathrsfs}
\usepackage{amsfonts}
\usepackage{amssymb}
\usepackage{amsxtra}
\usepackage{amsmath}
\usepackage{dsfont}
\usepackage{color}
\usepackage{graphicx}
\usepackage[english,polish]{babel}
\usepackage{enumerate}
\usepackage{mathtools}

\usepackage{newtxmath}
\usepackage{newtxtext}

\usepackage[margin=2.5cm, centering]{geometry}
\usepackage[colorlinks,citecolor=blue,urlcolor=blue,bookmarks=true]{hyperref}
\hypersetup{
pdfpagemode=UseNone,
pdfstartview=FitH,
pdfdisplaydoctitle=true,
pdfborder={0 0 0}, % No link borders
pdftitle={Spectral properties of some complex Jacobi matrices},
pdfauthor={Grzegorz Świderski},
pdflang=en-GB
}

\newcommand{\CC}{\mathbb{C}}
\newcommand{\ZZ}{\mathbb{Z}}
\newcommand{\sS}{\mathbb{S}}
\newcommand{\NN}{\mathbb{N}}
\newcommand{\RR}{\mathbb{R}}

\newcommand{\calC}{\mathcal{C}}

\newcommand{\calA}{\mathcal{A}}

\newcommand{\calX}{\mathcal{X}}
\newcommand{\calY}{\mathcal{Y}}

\newcommand{\calQ}{\mathcal{Q}}

\newcommand{\frakB}{\mathfrak{B}}
\newcommand{\frakX}{\mathfrak{X}}

\newcommand{\twisD}{\widetilde{\mathcal{D}}_1}
\newcommand{\twisV}{\widetilde{\mathcal{V}}_1}

\newcommand{\Id}{\operatorname{Id}}

\DeclareMathOperator{\Dom}{Dom}

\DeclareMathOperator{\sym}{sym}
\DeclareMathOperator{\tr}{tr}

\newcommand{\per}{\mathrm{per}}

\newcommand{\pl}[1]{\foreignlanguage{polish}{#1}}

\newcommand{\GL}{\operatorname{GL}}

\newcommand{\discr}{\operatorname{discr}}

\newcommand{\Amax}{A_\mathrm{max}}
\newcommand{\sigmaP}{\sigma_\mathrm{p}}
\newcommand{\sigmaEss}{\sigma_\mathrm{ess}}
\renewcommand{\Re}{\operatorname{Re}}

\newcommand{\lrparen}[1]{% \llrrparen{..}
  \left(\mkern-3mu #1 \mkern-3mu\right)}

\newtheorem{theorem}{Theorem}[section]

\newtheorem{proposition}[theorem]{Proposition}
\newtheorem{lemma}[theorem]{Lemma}
\newtheorem{corollary}[theorem]{Corollary}

\newtheorem*{theorem*}{Theorem}

\theoremstyle{plain}
\newcounter{thm}

\newtheorem{main_theorem}[thm]{Theorem}

\theoremstyle{definition}
\newtheorem{example}{Example}
\newtheorem{remark}[theorem]{Remark}

\numberwithin{equation}{section}

\usepackage[utf8]{inputenc}

\title{Spectral properties of some complex Jacobi matrices}
\author{Grzegorz Świderski}
\address{
	\pl{
	Grzegorz \'Swiderski \\
	Department of Mathematics \\
	KU Leuven \\
	Celestijnenlaan 200B box 2400 \\
	BE-3001 Leuven \\
	Belgium \&
	Mathematical Institute \\
	University of Wrocław \\
	pl. Grunwaldzki 2/4 \\
	50-384 Wrocław \\
	Poland
	}
}
\email{grzegorz.swiderski@kuleuven.be}

\keywords{Complex Jacobi matrix, asymptotics, generalised eigenvectors, continuous spectrum}

\subjclass[2010]{Primary: 47B36.}

\begin{document}
\selectlanguage{english}

\begin{abstract}
We study spectral properties of bounded and unbounded complex Jacobi matrices. In particular,
we formulate conditions assuring that the spectrum of the studied operators is continuous on
some subsets of the complex plane and we provide uniform asymptotics of their generalised eigenvectors. 
We illustrate our results by considering complex perturbations of real Jacobi matrices belonging to 
several classes: asymptotically periodic, periodically modulated and the blend of these two.
Moreover, we provide conditions implying existence of a unique closed extension. The method of
the proof is based on the analysis of a generalisation of shifted Tur\'an determinants
to the complex setting.
\end{abstract}

\maketitle

\section{Introduction}
For given two sequences $(a_n : n \in \NN_0)$ and $(b_n : n \in \NN_0)$ of complex numbers such that
for every $n \geq 0$ one has $a_n \neq 0$, we define the (complex) \emph{Jacobi matrix} by the formula
\begin{equation*}
	\calA =
	\begin{pmatrix}
         b_0 & a_0 & 0   & 0      &\ldots \\
         a_0 & b_1 & a_1 & 0       & \ldots \\
         0   & a_1 & b_2 & a_2     & \ldots \\
         0   & 0   & a_2 & b_3   &  \\
         \vdots & \vdots & \vdots  &  & \ddots
	\end{pmatrix}.
\end{equation*}
The action of $\calA$ on \emph{any} sequence of complex numbers is defined by the formal matrix multiplication. 
Let $A$ be the minimal operator associated with $\calA$. Specifically, by $A$ we mean the closure in 
$\ell^2(\NN_0)$ of the restriction of $\calA$ to the set of the sequences of finite support. Let us 
recall that
\[
	\langle x, y \rangle = \sum_{n=0}^\infty x_n \overline{y_n}, \quad
	\ell^2(\NN_0) = \big\{ x \in \CC^{\NN_0} \colon \langle x, x \rangle < \infty \big\}.
\]
Moreover, we define the maximal operator $\Amax$ by setting $\Amax x = \calA x$ for $x \in \Dom(\Amax)$, where
\[
	\Dom(\Amax) = \big\{ x \in \ell^2(\NN_0) : \calA x \in \ell^2(\NN_0) \big\}.
\]
The matrix $\calA$ is called \emph{proper} if $A = \Amax$ and \emph{improper} otherwise. 
The matrix $\calA$ is proper if the Carleman condition 
\begin{equation} \label{eq:1}
	\sum_{n=0}^\infty \frac{1}{|a_n|} = \infty
\end{equation}
is satisfied (see, \cite[Example 2.7]{Beckermann2001}). In the case when sequences $(a_n : n \in \NN_0)$ and $(b_n : n \in \NN_0)$ are real, the matrix $\calA$ is proper exactly when $A$ is self-adjoint.

The Jacobi matrix $\calA$ is \emph{symmetric} if $a_n \in \RR$ and $b_n \in \RR$ for any $n$. 
Otherwise $\calA$ is \emph{non-symmetric}.
In the symmetric case  we recover classical Jacobi matrices. This case has been studied thoroughly. 
Let us mention that it has applications in such areas as: spectral theory of self-adjoint operators 
(see, e.g. \cite{Dombrowski1985}), orthogonal polynomials,
approximation theory, numerical analysis, stochastic processes (see, e.g. \cite{Schoutens2000}), the moment problem and continued fractions (see, e.g. \cite{Simon1998}). For the applications and basic
properties of the non-symmetric case we refer to \cite{Beckermann2001} and to the references therein.

Spectral analysis of some classes of non-symmetric Jacobi matrices has been studied in several articles. In \cite{Arlinskii2006} there was studied complex one-rank perturbation of symmetric Jacobi matrices. In \cite{Beckermann1997} there was considered in detail the case when sequences defining 
$\calA$ are periodic. In \cite{Christiansen2017, Egorova2005, Golinskii2005, Golinskii2007, Golinskii2005a, Hansmann2011} there was considered the behaviour of the
point spectrum of compact perturbations of the periodic case. In \cite{Bourget2018} Mourre 
commutator method has been applied to the study of continuous spectrum of discrete 
Schr\"odinger operators. In \cite{Siegl2017} has been studied in detail an explicit example of 
unbounded complex Jacobi matrices. In particular, it was observed there a spectral phase transition
phenomenon. In \cite{Stampach2013, Stampach2017} a class of Jacobi matries with discrete spectrum has been studied. Finally, in \cite{Malejki2014} there was considered the point spectrum of some unbounded
complex Jacobi matrices.

The main problem in the study of spectral properties of the non-symmetric case lies in the fact that
in general $A$ is not normal. Hence, most of the tools and intuitions coming from the Spectral Theorem
is lost. It seems that there are few results concerning the continuous spectrum of non-symmetric
Jacobi matrices.  The aim of this article is to fill this gap.

Before we go further let us introduce some terminology. A sequence $(u_n \in \NN_0)$ is \emph{generalised 
eigenvector} associated with $z \in \CC$, if it satisfies the recurrence relation
\begin{equation} \label{eq:11}
	z u_n = a_n u_{n+1} + b_n u_n + a_{n-1} u_{n-1}, \quad (n \geq 1),
\end{equation}
with some non-zero initial condition $(u_0, u_1)$.
For any positive integer $N$ define \emph{$N$-step transfer matrix} by
\begin{equation} \label{eq:12}
	X_n(z) = \prod_{j=n}^{N+n-1} B_j(z) \quad \text{where} \quad
	B_j(z) =
	\begin{pmatrix}
		0 & 1 \\
		-\frac{a_{j-1}}{a_j} & \frac{z - b_j}{a_j}
	\end{pmatrix},
\end{equation}
then
\[
	\begin{pmatrix}
		u_{n+N-1} \\
		u_{n+N}
	\end{pmatrix}
	=
	X_n(z)
	\begin{pmatrix}
		u_{n-1} \\
		u_n
	\end{pmatrix}, \quad (n \geq 1).
\]
Given a compact set $K \subset \CC$, we say that the uniformly bounded
sequence of continuous mappings $Y_n : K \rightarrow \GL(2, \CC)$ belongs to 
$\twisD \big( K, \GL(2, \CC) \big)$ if\footnote{For any matrix $X$ we form the matrix $\overline{X}$ by taking the complex conjugation of every entry of $X$.}
\[
	\sum_{n \geq 1} \sup_{z \in K} \big\| Y_{n+1}(z) - \overline{Y_n(z)} \big\| < \infty,
\]

We are ready to state our main result.
\begin{main_theorem} \label{thm:A}
Let $N$ be a positive integer and $i \in \{0, 1, \ldots, N-1\}$. Suppose that
\begin{equation} \label{eq:57}
	\lim_{n \to \infty} X_{nN+i} = \calX, \qquad 
	\lim_{n \to \infty} \frac{a_{nN+i-1}}{|a_{nN+i-1}|} = \gamma, \qquad
	\lim_{n \to \infty} \frac{a_{nN+i-1}}{a_{(n+1)N+i-1}} = 1
\end{equation}
for some $\gamma$ and $\calX$. Let $K$ be a compact subset of\footnote{A discriminant of a $2 \times 2$ matrix $X$ is $\discr X = (\tr X)^2 - 4 \det X$.}
\begin{equation} \label{eq:59}
	\Lambda = \big\{ z \in \CC : \calX(z) \in \GL(2, \RR) \text{ and } \discr \calX(z) < 0 \big\}.
\end{equation}
If 
\begin{equation} \label{eq:58}
	\big( X_{nN+i} : n \in \NN \big) \in \twisD \big( K, \GL(2, \CC) \big),
\end{equation}
then there is a constant $c>1$ such that for every generalised eigenvector associated with $z \in K$
\[
	c^{-1} \big( |u_0|^2 + |u_1|^2 \big) \leq
	|a_{nN+i-1}| \big( |u_{nN+i-1}|^2 + |u_{nN+i}|^2 \big) \leq
	c \big( |u_0|^2 + |u_1|^2 \big).
\]
\end{main_theorem}

A simple consequence of Theorem~\ref{thm:A} is the following
\begin{main_theorem} \label{thm:B}
Let the hypotheses of Theorem~\ref{thm:A} be satisfied for some $i \in \{0, 1, \ldots, N-1\}$. If
\[
	\sum_{n=1}^\infty \frac{1}{|a_{nN+i-1}|} = \infty,
\]
then $\calA$ is proper\footnote{For any operator $X$ by $\sigma(X)$, $\sigmaEss(X)$ and $\sigmaP(X)$ we denote the spectrum, the essential spectrum and the point spectrum of $X$, respectively.}, $K \cap \sigmaP(A) = \emptyset$ and $K \subset \sigma(A)$. Conversely, 
if 
\[
	\sum_{n=0}^\infty \frac{1}{|a_n|} < \infty,
\]
and the hypotheses are satisfied for \emph{every} $i \in \{ 0, 1, \ldots, N-1 \}$,
then $\calA$ is improper, $\sigmaEss(A) = \emptyset$, $\sigma(A) = \CC$ and $\sigmaP(\Amax) = \CC$.
\end{main_theorem}

Let us present some of the ideas of the proof of Theorem~\ref{thm:A}. For any generalised eigenvector
$u$ we define (generalised) \emph{$N$-shifted Tur\'an determinant} by the formula
\[
	S_n = 
	\Re 
	\Big( 
	\overline{\gamma} a_{n+N-1} 
	\big( \overline{u_n} u_{n+N-1} - \overline{u_{n-1}} u_{n+N}\big) 
	\Big).
\]
Observe that in the symmetric case if $z \in \RR$ and both $u_0, u_1$ are real, then the definition
reduces to the classical $N$-shifted Tur\'an determinant
\[
	a_{n+N-1} 
	\det
	\begin{pmatrix}
		u_{n+N-1} & u_{n-1} \\
		u_{n+N} & u_n
	\end{pmatrix}.
\]
We prove Theorem~\ref{thm:A} by careful analysis of the sequence $(S_{nN+i} : n \in \NN)$. 
In particular, we prove that this sequence is convergent uniformly on $K$ to a function of definite
sign and without zeros. In the setup of symmetric Jacobi matrices this approach turned out to
be fruitful and allowed even to recover the spectral measure of $A$ and derive pointwise asymptotics of
their formal eigenvector, see \cite{SwiderskiTrojan2019} for details. The method used in \cite{SwiderskiTrojan2019} is based on the diagonalisation of transfer matrices. This idea 
seems to be difficult to apply in our setting. Instead we extend some ideas from our recent articles \cite{block2018, periodicI}.

The structure of the article is as follows. In Section~\ref{sec:2} we fix notation and collected some
basic definitions used in the rest of the article. The relations between generalised eigenvectors
and spectral properties of complex Jacobi matrices is studied in Section~\ref{sec:3}. 
In Section~\ref{sec:4} we derive some general properties of quadratic forms on $\CC^2$. 
Section~\ref{sec:5} is devoted to the study of Tur\'an determinants. In the last section, we give
several classes of Jacobi matrices to which our results can be applied. 

\subsection*{Notation}
By $\ZZ$ and $\NN$ we denote the integers and the positive integers, respectively. 
The non-negative integers are denoted by $\NN_0$. The complex numbers are denoted by $\CC$, whereas $\sS^1$ is the set of the complex numbers of modulus $1$. Moreover, by $c$ we denote generic positive constants whose value may change from line to line.

\subsection*{Acknowledgement}
The author would like to thank an anonymous referee for helpful suggestions.
The author was partialy supported by the Foundation for Polish Science (FNP) and 
by long term structural funding: Methusalem grant of the Flemish Government. Part of this work
has been done while the author was employed by the Polish Academy of Sciences.

\section{Preliminaries} \label{sec:2}
In this section we fix notation and we collected some basic definitions used in the rest of the article.

\subsection{Matrices}
By $M_d(\CC)$ we mean the space of $d$ by $d$ matrices with complex entries equipped with the
operator norm. For brevity we sometimes identify complex numbers with $M_1(\CC)$.

For a~sequence of matrices $(X_n : n \in \NN)$ and $n_0, n_1 \in \NN$ we set
\[
      \prod_{k=n_0}^{n_1} X_k = 
       \begin{cases} 
       X_{n_1} X_{n_1 - 1} \cdots X_{n_0} & n_1 \geq n_0, \\
       \Id & \text{otherwise.}
       \end{cases}
\]

For any matrix $X$, we form the matrix $\overline{X}$ by taking complex conjugation of 
every entry of $X$, i.e.
\begin{equation} \label{eq:2}
	\overline{X}_{ij} = \overline{X_{ij}}, \quad (i,j = 1,2, \ldots, d).
\end{equation}
Moreover, we denote by $X^t$ and $X^*$ the transpose and the Hermitian transpose of $X$, respectively.

For every $X \in M_d(\CC)$,
\begin{equation} \label{eq:56}
	\| X \| \leq \| X \|_2 \leq \| X \|_1,
\end{equation}
where $\| X \|_t$ the $t$-norm of the matrix considered as the element of $\CC^{d^2}$.

For any matrix $X$ we define its \emph{Hermitian symmetrisation} by
\[
	\sym [X] = \frac{1}{2} \big( X + X^* \big).
\]
Direct computation shows that for any matrix $Y$ one has
\begin{equation} \label{eq:3}
	Y^* \sym [X] Y = \sym [Y^* X Y]
\end{equation}
and for every $\alpha \in \RR$
\begin{equation} \label{eq:4}
	\sym [\alpha X + Y] = \alpha \sym [X] + \sym [Y].
\end{equation}
Moreover,
\begin{equation} \label{eq:5}
	\big\| \sym [X] \big\| \leq \| X \|.
\end{equation}

\subsection{Twisted Stolz class}
Let $V$ be a normed vector space equipped with an additive mapping preserving the norm, which will be
called \emph{conjugation}. Specifically, we demand that for every $x, y \in V$ one has
\begin{enumerate}[(a)]
\item $\begin{aligned}[b]
	\overline{x + y} = \overline{x} + \overline{y},
\end{aligned}$

\item $\begin{aligned}[b]
	\| \overline{x} \| = \| x \|
\end{aligned}$
\end{enumerate}

We define \emph{twisted total variation} of a sequence of vectors $x = (x_n : n \geq M)$ from $V$ by
\begin{equation} \label{eq:6}
	\twisV(x) = \sum_{n=M}^\infty \big\| x_{n+1} - \overline{x_n} \big\|.
\end{equation}
We say that $x$ belongs to \emph{twisted Stolz class} $\twisD(V)$ if $x$ is bounded and 
$\twisV(x) < \infty$. 

The following proposition collects some properties of the twisted total variation. 
The proof is straightforward.
\begin{proposition} \label{prop:1}
Twisted total variation has the following properties.
\begin{enumerate}[(a)]
\item For every $x, y \in V$ one has 
$\begin{aligned}[b]
	\twisV(x+y) \leq \twisV(x) + \twisV(y)
\end{aligned}$

\item If $x \in \twisD(V)$ and $x_{\infty} := \lim_{n \to \infty} x_{n}$ exists, then 
$x_{\infty} = \overline{x_\infty}$.
\end{enumerate}
\end{proposition}

\begin{proposition} \label{prop:2}
If $V$ is a normed algebra and the conjugation is a homomorphism with respect to the multiplication, 
then
\begin{equation} \label{eq:7}
	\twisV(x_n y_n : n \geq M) \leq
	\sup_{n \geq M} \| x_n \| \cdot \twisV(y_n : n \geq M) +
	\sup_{n \geq M} \| y_n \| \cdot \twisV(x_n : n \geq M)
\end{equation}
and if for all $n \geq M$ the element $x_n$ is invertible, then
\begin{equation} \label{eq:8}
	\twisV(x_n^{-1} : n \geq M) \leq
	\sup_{n \geq M} \| x_n^{-1} \|^2 \cdot \twisV(x_n : n \geq M).
\end{equation}
\end{proposition}
\begin{proof}
Observe that
\[
	x_{n+1} y_{n+1} - \overline{x_n y_n} = 
	(x_{n+1} - \overline{x_n}) y_{n+1} + 
	\overline{x_n} (y_{n+1} - \overline{y_n}).
\]
Hence,
\[
	\| x_{n+1} y_{n+1} - \overline{x_n y_n} \| \leq 
	\| x_{n+1} - \overline{x_n} \| \cdot \| y_{n+1} \| + 
	\| x_{n} \| \cdot \| y_{n+1} - \overline{y_n} \|.
\]
Consequently,
\[
	\| x_{n+1} y_{n+1} - \overline{x_n y_n} \| \leq 
	\sup_{m \geq M} \| y_{m} \| \cdot \| x_{n+1} - \overline{x_n} \| + 
	\sup_{m \geq M} \| x_{m} \| \cdot \| y_{n+1} - \overline{y_n} \|.
\]
Summing by $n$ the formula \eqref{eq:7} follows.

To prove \eqref{eq:8}, observe
\[
	x_{n+1}^{-1} - \overline{x_n^{-1}} = 
	x_{n+1}^{-1} ( \overline{x_n} - x_{n+1} ) \overline{x_n^{-1}}.
\]
Hence
\[
	\Big\| x_{n+1}^{-1} - \overline{x_n^{-1}} \Big\| \leq
	\Big\| x_{n+1}^{-1} \Big\| \cdot 
	\Big\| \overline{x_n} - x_{n+1} \Big\| \cdot
	\Big\| \overline{x_n^{-1}} \Big\|
\]
and the formula \eqref{eq:8} readily follows.
\end{proof}

Let $K$ be a compact subset of the complex plane. In the sequel we are going to use mostly the
Banach algebra $V = C \big( K, M_d(\CC) \big)$ consisting of continuous mappings on $K$ with
values in $M_d(\CC)$. The associated norm is defined by
\begin{equation} \label{eq:9}
	\| f \|_\infty = \sup_{x \in K} \| f(x) \|,
\end{equation}
where $\| \cdot \|$ is the operator norm. For any $f \in V$ we define $\overline{f}$ by
\begin{equation} \label{eq:10}
	\overline{f}(x) = \overline{f(x)}, \quad (x \in K),
\end{equation}
where the matrix $\overline{f(x)}$ is defined in \eqref{eq:2}.

\section{Generalised eigenvectors} \label{sec:3}
For a~number $z \in \CC$, a~non-zero sequence $u = (u_n \colon n \in \NN_0)$ 
will be called a~\emph{generalised eigenvector} provided that it satisfies
the recurrence relation \eqref{eq:11}.
For each non-zero $\alpha \in \CC^2$ there is a unique generalised
eigenvector $u$ such that $(u_0, u_1)^t = \alpha$.

\begin{proposition} \label{prop:3}
Let $z \in \CC$. The sequence $u$ satisfies $\calA u = z u$ if and only if
\begin{equation} \label{eq:14}
	\begin{gathered}
		u_0 \in \CC, \quad u_1 = \frac{z - b_0}{a_0} u_0, \\
		a_{n-1} u_{n-1} + b_n u_n + a_n u_{n+1} = z u_n, \quad (n \geq 1).
	\end{gathered}
\end{equation}
\end{proposition}
\begin{proof}
	It immediately follows from the direct computations.
\end{proof}

A matrix $\calA$ is called \emph{determinate} if for all $z \in \CC$ there exists a generalised eigenvector $u$ associated with $z$ such that $u \notin \ell^2(\NN_0)$. A matrix $\calA$ is called \emph{indeterminate} if it is not determinate.

\begin{remark} \label{rem:1}
If there is $z \in \CC$ and a generalised eigenvector $u$ associated with $z$ such that 
$u \notin \ell^2(\NN_0)$, then $\calA$ is determinate (see the discussion after 
\cite[Definition 2.5]{Beckermann2001}).
\end{remark}

The following Theorem has been proved in \cite[Theorem 2.1]{Beckermann2004}.
\begin{theorem}[Beckermann \& Castro Smirnova] \label{thm:3}
A matrix $\calA$ is determinate if and only if it is proper.
\end{theorem}

\begin{corollary} \label{cor:1}
Let $z_0 \in \CC$. If every generalised eigenvector associated with $z_0$ belongs to $\ell^2(\NN_0)$,
then the matrix $\calA$ is improper. Moreover, $\sigmaEss(A) = \emptyset$, $\sigma(A) = \CC$ and 
$\sigmaP(\Amax) = \CC$.
\end{corollary}
\begin{proof}
Remark \ref{rem:1} implies that $\calA$ is indeterminate. Hence, by Theorem~\ref{thm:3} $\calA$ 
is improper.
By \cite[Theorem 2.11(a)]{Beckermann2001}
we obtain
\[
	\sigmaEss(A) = \emptyset \quad \text{and} \quad
	\sigma(A) = \CC.
\]
Let $z \in \CC$. Let $u=(u_n : n \in \NN_0)$ satisfy \eqref{eq:14} with $u_0=1$. Since 
$u \in \ell^2(\NN_0)$ we obtain $z \in \sigmaP(\Amax)$. Thus, $\sigmaP(\Amax) = \CC$ and
the proof is complete.
\end{proof}
	
\begin{proposition} \label{prop:4}
Let $z_0 \in \CC$. If every generalised eigenvector $u$ associated with $z_0$ does 
not belong to $\ell^2(\NN_0)$ then the matrix $\calA$ is proper, $z_0 \notin \sigmaP(\Amax)$ and 
$z_0 \in \sigma(\Amax)$.
\end{proposition}
\begin{proof}
Remark \ref{rem:1} implies that $\calA$ is determinate. Hence, Theorem \ref{thm:3} implies that $\calA$
is proper.

Let the non-zero sequence $u$ be such that $\calA u = z_0 u$, then by Proposition~\ref{prop:3}, 
the sequence $u$ is a generalised eigenvector associated with $z_0$. By the assumption 
$u \notin \ell^2(\NN_0)$. Therefore, $u \notin \Dom(\Amax)$, and consequently, 
$z_0 \notin \sigmaP(\Amax)$.
	      
Observe that if there is a vector $u$ such that $(\calA - z_0 \Id) u = \delta_0$, then it has to satisfy
the following recurrence relation
\[
	\begin{gathered}
		b_0 u_0 + a_0 u_1 = z_0 u_0 + 1 \\
	    a_{n-1} u_{n-1} + b_n u_n + a_n u_{n+1} = z_0 u_n \quad (n \geq 1)
	\end{gathered}
\]
Hence $u$ is a generalised eigenvector associated with $z_0$, thus $u \notin \ell^2(\NN_0)$. 
Therefore, $u \notin \Dom(\Amax)$, and consequently, the operator $\Amax - z_0 \Id$ is not surjective, 
i.e. $z_0 \in \sigma(\Amax)$.
\end{proof}

\section{Uniform non-degeneracy of quadratic forms} \label{sec:4}
Let $K$ be a compact subset of $\CC$. Suppose that for each $z \in K$ there is a sequence
$Q^z = (Q_n^z : n \in \NN)$ of quadratic forms on $\CC^2$. We say that
$\{ Q^z : z \in K \}$ is \emph{uniformly non-degenerated} if there are 
$c_1> 0, c_2 > 0$ and $M \geq 1$ such that for all $v \in \CC^2$, $z \in K$ and $n \geq M$
\[
	c_1 \| v \|^2 \leq |Q_n^z(v)| \leq c_2 \| v \|^2.
\]

In the rest of this article we will use the following matrix
\begin{equation} \label{eq:49}
	E = 
	\begin{pmatrix}
		0 & -1 \\
		1 & 0	
	\end{pmatrix}.
\end{equation}

\begin{proposition} \label{prop:5}
Let $K$ be a compact subset of $\CC$. Let $\{ Q^z : z \in K \}$ be a family of quadratic forms on 
$\CC^2$ given by
\[
	Q^z_n(v) = \big\langle \sym[E Y_n(z)] v, v \big\rangle,
\]
where each $Y_n$ is continuous. If
\[
	\lim_{n \to \infty} \sup_{z \in K} \| Y_n(z) - \calY(z) \| = 0,
\]
where for any $z \in K$
\begin{equation} \label{eq:50}
	\calY(z) \in \GL( 2, \RR) \quad \text{ and } \quad \discr \calY(z) < 0,
\end{equation}
then $\{ Q^z : z \in K \}$ is uniformly non-degenerated on $K$.
\end{proposition}
\begin{proof}
Define
\[
	\calQ^z(v) = 
	\big\langle \sym[E \calY(z)] v, v \big\rangle.
\]
Since the matrix $\sym[E \calY(z)]$ is self-adjoint, it has two real eigenvalues 
$\lambda_1(z) \leq \lambda_2(z)$. Moreover, it has real entries for any $z \in K$. Thus, by direct
computation one can verify that
\[
	\det \Big( \sym \big[ E \calY(z) \big] \Big) = -\frac{1}{4} \discr \big( \calY(z) \big).
\]
Therefore, by \eqref{eq:50} and the continuity of $\calY$ on $K$, we obtain that $\lambda_1(z)$ and 
$\lambda_2(z)$ are bounded and of the same sign for every $z \in K$. Consequently, there are 
constants $c_1>0, c_2>0$
such that for any $v \in \CC^2$
\begin{equation} \label{eq:21a}
	c_1 \|v\|^2 \leq | \calQ^z(v) | \leq c_2 \|v\|^2.
\end{equation}

Fix $\epsilon > 0$. Let $M$ be such that for every $n \geq M$
\[
	\sup_{z \in K} \| Y_n(z) - \calY(z) \| < \epsilon.
\]
Observe
\begin{equation} \label{eq:21}
	|\calQ^z(v)| - |Q_n^z(v) - \calQ^z(v)| 
	\leq 
	|Q^z_n(v)| 
	\leq 
	|\calQ^z(v)| + |Q^z_n(v) - \calQ^z(v)|.
\end{equation}
Thus, by \eqref{eq:21a}
\[
	(c_1 - \epsilon) \| v \|^2 \leq |\calQ^z_n(v)| \leq (c_2 + \epsilon) \| v \|^2.
\]
The conclusion follows by taking $\epsilon$ sufficiently small.
\end{proof}

\section{Shifted Tur\'an determinants} \label{sec:5}
In this section we define and study the convergence of Tur\'an determinants.

\subsection{Definitions and basic properties}
Fix a~positive integer $N$, a~Jacobi matrix $A$ and $\gamma \in \sS^1$. 
Let us define a~sequence of quadratic forms $Q^{z, \gamma}$ on $\CC^2$ by the formula
\begin{equation} \label{eq:15}
	Q_n^{z, \gamma}(v) = 
	\bigg\langle 
	\sym 
	\bigg[ 
		\frac{a_{n+N-1}}{\gamma |a_{n+N-1}|}
		E X_n(z)
	\bigg] v, v 
	\bigg\rangle,
\end{equation}
where $X_n$ and $E$ are defined in \eqref{eq:12} and \eqref{eq:49}, respectively.
Then we define the $N$-shifted Turán determinants by
\begin{equation} \label{eq:17}
	S_n^{\gamma}(\alpha, z) = |a_{n+N-1}| Q_n^{z, \gamma} 
	\lrparen{
	\begin{pmatrix} 
		u_{n-1} \\
		u_n
	\end{pmatrix}},
\end{equation}
where $u$ is the generalised eigenvector corresponding to $z \in \CC$ such that 
$(u_0, u_1)^t = \alpha \in \CC^2$. 

The study of the sequence $(S_n^\gamma : n \in \NN)$ is motivated by the following theorem, whose proof is analogous to the proof
of \cite[Theorem 7]{block2018}. We include it for the sake of completeness.
\begin{theorem}
\label{thm:1}
Let $N$ be a positive integer and $i \in \{0, 1, \ldots, N-1 \}$. Let $K$ be a compact set. 
Assume that the family $\big\{ \big(Q^{z, \gamma}_{nN+i} : n \in \NN \big) : z \in K \big\}$ defined in
\eqref{eq:15} is uniformly non-degenerated. Suppose that there are $c'_1>0, c'_2>0$ and $M' \geq 1$ such that 
for all $\alpha \in \CC^2$ such that $\| \alpha \| = 1$, $z \in K$ and $n \geq M'$ 
\begin{equation}
	\label{eq:24}
	c'_1 \leq | S_n^\gamma(\alpha, z) | \leq c'_2.
\end{equation}
Then there is $c > 1$ such that for all $z \in K$, $n \geq 1$ and for every generalised
eigenvector $u$ corresponding to $z$
\[ 
	c^{-1} \big( |u_0|^2 + |u_1|^2 \big) 
	\leq |a_{(n+1)N+i-1} | \big( |u_{nN+i-1}|^2 + | u_{nN+i}|^2 \big) 
	\leq c \big( |u_0|^2 + |u_1|^2 \big).
\]
\end{theorem}
\begin{proof}
Let $z \in K$ and let $u$ be a generalised eigenvector corresponding to $z$ such that $\alpha = (u_0, u_1)^t$. Observe that
\begin{equation} \label{eq:51}
	S^\gamma_n ( \alpha, z ) = \| \alpha \|^2 S^\gamma_n \bigg( \frac{\alpha}{\|\alpha\|}, z \bigg)
\end{equation}
Hence, it is enough to prove the conclusion for $\|\alpha\| = 1$. Since the family 
$\big\{ \big(Q^{z, \gamma}_{nN+i} : n \in \NN \big) : z \in K \big\}$ is uniformly non-degenerated, 
there are $c_1 > 0, c_2 > 0$ and $M \geq M'$ such that for all $n \geq M$
\[
	c_1 | a_{(n+1)N+i-1} |
	\big( |u_{nN+i-1}|^2 + |u_{nN+i}|^2 \big) \leq
	\big| S^{\gamma}_{nN+i}(\alpha, z) \big| \leq
	c_2 | a_{(n+1)N+i-1} |
	\big( |u_{nN+i-1}|^2 + |u_{nN+i}|^2 \big).
\]
Hence, by \eqref{eq:24}
\[
	c_1' c_2^{-1} \leq
	| a_{(n+1)N+i-1} |
	\big( |u_{nN+i-1}|^2 + |u_{nN+i}|^2 \big) \leq
	c_2' c_1^{-1}
\]
for any $n \geq M$. Since each $u_n$ is a continuous function of $z$ we can find another constant $c>1$ such that
\[
	c^{-1} \leq
	| a_{(n+1)N+i-1} |
	\big( |u_{nN+i-1}|^2 + |u_{nN+i}|^2 \big) \leq
	c.
\]
for any $n \geq 1$. In view of \eqref{eq:51} the conclusion follows.
\end{proof}

\subsection{The proof of the convergence}
In this section we are going to show that $(S_{nN+i} : n \in \NN)$ is uniformly convergent to
some function.

\begin{proposition} \label{prop:6}
An alternative formula for $S_n^\gamma$ is
\begin{equation} \label{eq:19}
	S_n^\gamma(\alpha, z) = |a_{n+N-1}| \widetilde{Q}_n^{z, \gamma}
	\lrparen{
		\begin{pmatrix}
			u_{n+N-1} \\
			u_{n+N}
		\end{pmatrix}
	},
\end{equation}
where
\begin{equation} \label{eq:20}
	\widetilde{Q}_n^{z, \gamma}(v) = 
	\bigg\langle
		\sym 
		\bigg[
			\frac{a_{n+N-1}}{\gamma |a_{n+N-1}|} 
			\bigg( \frac{a_{n+N-1}}{a_{n-1}} \bigg)^* E \overline{X_n(z)}
		\bigg]
		v, v
	\bigg\rangle.
\end{equation}
\end{proposition}
\begin{proof}
By \eqref{eq:17} and $\gamma^{-1} = \overline{\gamma}$ one has
\[
	S_n^\gamma(\alpha, z) = 
	\bigg\langle
		\sym 
		\big[
			\overline{\gamma} a_{n+N-1} E X_n(z)
		\big]
		\begin{pmatrix}
			u_{n-1} \\
			u_n
		\end{pmatrix},
		\begin{pmatrix}
			u_{n-1} \\
			u_n
		\end{pmatrix}
	\bigg\rangle.
\]
Thus,
\begin{align*}
	S_n^\gamma(\alpha, z) 
	&=
	\bigg\langle
		\sym 
		\big[
			\overline{\gamma} a_{n+N-1}E X_n(z)
		\big]
		X_n^{-1}(z)
		\begin{pmatrix}
			u_{n+N-1} \\
			u_{n+N}
		\end{pmatrix},
		X_{n}^{-1}(z)
		\begin{pmatrix}
			u_{n+N-1} \\
			u_{n+N}
		\end{pmatrix}
	\bigg\rangle \\
	&=
	\bigg\langle
		\big[ X_{n}^{-1}(z) \big]^*
		\sym 
		\big[
			\overline{\gamma} a_{n+N-1} E X_n(z)
		\big]
		X_n^{-1}(z)
		\begin{pmatrix}
			u_{n+N-1} \\
			u_{n+N}
		\end{pmatrix},
		\begin{pmatrix}
			u_{n+N-1} \\
			u_{n+N}
		\end{pmatrix}
	\bigg\rangle.
\end{align*}
Hence, by \eqref{eq:3}
\[
	S_n^\gamma(\alpha, z) = 
	\bigg\langle
		\sym 
		\Big[
			\overline{\gamma} a_{n+N-1} 
			\big[ X_{n}^{-1}(z) \big]^* E
		\Big]
		\begin{pmatrix}
			u_{n+N-1} \\
			u_{n+N}
		\end{pmatrix},
		\begin{pmatrix}
			u_{n+N-1} \\
			u_{n+N}
		\end{pmatrix}
	\bigg\rangle.
\]
By direct computations one can verify that for any $X \in M_2(\CC)$
\[
	\det(\overline{X}) \big( X^{-1} \big)^* E = E \overline{X} = 
	\begin{pmatrix}
		-\overline{X_{21}} & -\overline{X_{22}} \\
		\overline{X_{12}} & \overline{X_{11}}
	\end{pmatrix}.
\]
Thus
\[
	\overline{\gamma} a_{n+N-1} \big[ X_{n}^{-1}(z) \big]^* E = 
	\overline{\gamma} a_{n+N-1} \bigg( \frac{a_{n+N-1}}{a_{n-1}} \bigg)^* E \overline{X_n(z)}
\]
and the formula \eqref{eq:19} follows.
\end{proof}

The next lemma provides the main algebraic part of our main result.
\begin{lemma} \label{lem:1}
Let $u$ be a~generalised eigenvector associated with $z \in \CC$ and $\alpha \in \CC^2$. Then
\begin{equation} \label{eq:25}
	|S_{n+N}^{\gamma}(\alpha, z) - S_{n}^{\gamma}(\alpha, z)| \leq
	\bigg\| a_{n+2N-1} X_{n+N}(z) -
	a_{n+N-1} \bigg( \frac{a_{n+N-1}}{a_{n-1}} \bigg)^* \overline{X_{n}(z)} \bigg\|
	\Big( |u_{n+N-1}|^2 + |u_{n+N}|^2 \Big).
\end{equation}
\end{lemma}
\begin{proof}
By Proposition~\ref{prop:6}
\begin{equation} \label{eq:26}
	S_{n+N}^\gamma(\alpha, z) - S_n^\gamma(\alpha, z) =
	\bigg\langle
		\sym \big[ \overline{\gamma} C_n(z) \big]
		\begin{pmatrix}
			u_{n+N-1} \\
			u_{n+N}
		\end{pmatrix},
		\begin{pmatrix}
			u_{n+N-1} \\
			u_{n+N}
		\end{pmatrix}
	\bigg\rangle,
\end{equation}
where
\begin{equation} \label{eq:27}
	C_n(z) = a_{n+2N-1} E X_{n+N}(z) - a_{n+N-1} \bigg( \frac{a_{n+N-1}}{a_{n-1}} \bigg)^* E \overline{X_n(z)}.
\end{equation}
Therefore, by the Schwarz inequality and \eqref{eq:5}
\[
	|S_{n+N}^\gamma(\alpha, z) - S_n^\gamma(\alpha, z)| \leq
	\| C_n(z) \| \big( |u_{n+N-1}|^2 + |u_{n+N}|^2 \big)
\]
and the formula \eqref{eq:25} follows. The proof is complete.
\end{proof}

We are ready to prove the main result of this article.
\begin{theorem} \label{thm:2}
Let $N$ be a positive integer, $i \in \{0, 1, \ldots, N-1 \}$ and $\gamma \in \sS^1$.
Suppose that $K \subset \CC$ and $\Omega \subset \CC^2 \setminus \{ (0,0) \}$ are compact connected sets. 
Assume that
\begin{enumerate}[(a)]
\item
$\begin{aligned}[b]
	\bigg( \frac{a_{(n+1)N+i-1}}{a_{nN+i-1}} X_{nN+i} : n \in \NN \bigg) \in \twisD \big( K, \GL(2, \CC) \big) 
\end{aligned}$ \label{thm:2:eq:1} 

\item the family $\big\{ \big( \widetilde{Q}_{nN+i}^{z, \gamma} : n \in \NN \big) : z \in K \big\}$
defined in \eqref{eq:20} is uniformly non-degenerated. \label{thm:2:eq:2} 
\end{enumerate}
Then the limit
\begin{equation} \label{eq:28}
	g(\alpha, z) = \lim_{n \to \infty} S_{nN+i}^\gamma(\alpha, z) \qquad (\alpha \in \Omega,\ z \in K)
\end{equation}
exists, where the sequence $(S_n^\gamma : n \geq 1)$ is defined in \eqref{eq:17}. Moreover, $|g|$ is a strictly positive continuous 
function and the convergence in \eqref{eq:28} is uniform on $\Omega \times K$.
\end{theorem}
\begin{proof}
We are going to show \eqref{eq:28} and the existence of $c_1 > 0$, $c_2>0$ and $M > 0$ such that
\begin{equation} \label{eq:29}
	c_1 \leq \big| S^\gamma_{nN+i}(\alpha, \lambda) \big| \leq c_2
\end{equation}
for all $\alpha \in \Omega$, $z \in K$ and $n > M$.

Given a generalised eigenvector corresponding to $z \in K$ such that $(u_0, u_1)^t = \alpha \in \Omega$,
we can easily see that each $u_n$, considered as a~function of $\alpha$ and $z$,
is continuous on $\Omega \times K$. As a~consequence, the function $S_n^\gamma$ is continuous on
$\Omega \times K$. Since $\big\{ \big( \widetilde{Q}^{z, \gamma}_{nN+i} : n \in \NN \big) : 
z \in K \big\}$ is uniformly non-degenerated, then by \eqref{eq:19} 
there is $M > 0$ such that for each $n \geq M$ the function $S_{nN+i}^\gamma$ has no zeros and has 
the same sign for all $z \in K$
and $\alpha \in \Omega$. Otherwise, by the connectedness of $\Omega \times K$, there would be 
$\alpha \in \Omega$ and $z \in K$ such that $S_{nN+i}^\gamma(\alpha, z) = 0$, which would contradict 
the non-degeneracy of $Q_{nN+i}^{z, \gamma}$. 

Thus, in order to prove \eqref{eq:28} and \eqref{eq:29}, it is enough to show that
\begin{equation} \label{eq:30}
	\sum_{n = M}^\infty \sup_{\alpha \in \Omega} \sup_{\lambda \in K} |F_n(\alpha, \lambda)| 
	< \infty,
\end{equation}
where $(F_n : n \geq M)$ is a sequence of functions on $\Omega \times K$ 
defined by
\[
	F_n = \frac{S_{(n+1)N+i}^\gamma - S_{nN+i}^\gamma}{S_{nN+i}^\gamma}.
\]
Indeed,
\begin{equation} \label{eq:31}
	\prod_{j=M}^{k-1} (1 + F_{j}) = \prod_{j=M}^{k-1} \frac{S^\gamma_{(j+1)N+i}}{S^\gamma_{jN+i}} 
	= \frac{S^\gamma_{kN+i}}{S^\gamma_{MN+i}}
\end{equation}
and the condition \eqref{eq:30} implies that the product \eqref{eq:31} is convergent uniformly
on $\Omega \times K$ to a continuous function of a definite sign. This implies \eqref{eq:29}.

It remains to prove \eqref{eq:30}. Since $\big\{ \big( \widetilde{Q}_{nN+i}^{\gamma, z} : 
n \in \NN \big) : z \in K \big\}$ is uniformly non-degenerated, we have
\[
	|S^\gamma_{nN+i}(\alpha, z)|
	\geq
	c^{-1}
	|a_{(n+1)N+i-1}| 
	\big( |u_{(n+1)N+i-1}|^2 + |u_{(n+1)N+i}|^2 \big)
\]
for all $n \geq M$, $\alpha \in \Omega$ and $\lambda \in K$. Hence, by Lemma~\ref{lem:1}
\[
	\big|F_n(\alpha, z) \big|
	\leq
	c
	\bigg\| \frac{a_{(n+2)N+i-1}}{a_{(n+1)+i-1}} X_{(n+1)N+i-1}(z) -
	\bigg( \frac{a_{(n+1)N+i-1}}{a_{nN+i-1}} \bigg)^* \overline{X_{nN+i}(z)} \bigg\|.
\]
Hence, \eqref{thm:2:eq:1} implies \eqref{eq:30}. The proof is complete.
\end{proof}

Now, we can give proofs of Theorem~\ref{thm:A} and Theorem~\ref{thm:B}.
\begin{proof}[Proof of Theorem~\ref{thm:A}]
Observe that 
\[
	\lim_{n \to \infty}
	\frac{a_{(n+1)N+i-1}}{\gamma |a_{(n+1)N+i-1}|}  
	\bigg( \frac{a_{(n+1)N+i-1}}{a_{nN+i-1}} \bigg)^* \overline{X_{nN+i}(z)}
	=
	\overline{\calX(z)} = \calX(z)
\]
uniformly with respect to $z \in K$. Hence, by Proposition~\ref{prop:5} and \ref{prop:6} the family
$\big\{ \big( \widetilde{Q}_{nN+i}^{z, \gamma} : n \in \NN \big) : z \in K \big\}$ is uniformly 
non-degenerated.

Observe that
\[
	(X_{nN+i} : n \in \NN) \in \twisD \big( K, \GL(2, \CC) \big)
\]
implies that every entry of $X_{nN+i}$ belongs to $\twisD(K, \CC)$.
We have
\[
	\frac{a_{nN+i-1}}{a_{(n+1)N+i-1}} = \det X_{nN+i}.
\]
Thus, by Proposition~\ref{prop:2}
\[
	\bigg( \frac{a_{nN+i-1}}{a_{(n+1)N+i-1}} : n \in \NN \bigg) \in \twisD(K, \CC).
\]
Since this sequence tends to $1$, again by Proposition~\ref{prop:2},
\[
	\bigg( \frac{a_{(n+1)N+i-1}}{a_{nN+i-1}} : n \in \NN \bigg) \in \twisD(K, \CC)
\]
and consequently, also
\[
	\bigg( \frac{a_{(n+1)N+i-1}}{a_{nN+i-1}} X_{nN+i} : n \in \NN \bigg) \in 
	\twisD \big(K, \GL(2, \CC) \big).
\]
Hence, the hypotheses of Theorem~\ref{thm:2} are satisfied. Finally, the conclusion follows from
Theorem~\ref{thm:1}.
\end{proof}

\begin{proof}[Proof of Theorem~\ref{thm:B}]
The conclusion follows from Theorem~\ref{thm:A} combined with Corollary~\ref{cor:1} and Proposition~\ref{prop:4}.
\end{proof}

\section{Applications} \label{seq:6}
In this section we present applications of the main results of this article. To simplify the exposition let us first introduce some notation. For any positive integer $N$, we say that a complex sequence $(x_n : n \in \NN)$ belongs to $\twisD^N(\CC)$ if for every $i \in \{0, 1, \ldots, N-1 \}$
\[
	(x_{nN+i} : n \in \NN) \in \twisD(\CC).
\] 

The following proposition will be used repeatedly in the rest of this section.
\begin{proposition} \label{prop:7}
Let $N$ a positive integer and $\gamma \in \sS^1$. Suppose that for some $i \in \{0, 1, \ldots, N-1 \}$,
\[
	\bigg( \frac{a_{nN+i-1}}{a_{nN+i}} : n \in \NN \bigg),\ 
	\bigg( \frac{b_{nN+i}}{a_{nN+i}} : n \in \NN \bigg),\ 
	\bigg( \frac{\gamma}{a_{nN+i}} : n \in \NN \bigg) \in \twisD(\CC).
\]
Then for every compact $K \subset \gamma \RR$,
\[
	(B_{nN+i} : n \in \NN) \in \twisD \big(K, \GL(2, \CC) \big).
\]
\end{proposition}
\begin{proof}
Let $K$ be a compact subset of $\gamma \RR$. Let $z \in K$, then $z = \gamma x$ for some $x \in \RR$.
Since
\[
	B_{(n+1)N+i}(z) - \overline{B_{nN+i}(z)} = 
	\begin{pmatrix}
		0 & 0 \\
		\Big( \frac{a_{nN+i-1}}{a_{nN+i}} \Big)^* - \frac{a_{(n+1)N+i-1}}{a_{(n+1)N+i}}& 
		x \Big[ \frac{\gamma}{a_{(n+1)N+i}} - \Big( \frac{\gamma}{a_{nN+i}} \Big)^* \Big] + 
		\Big( \frac{b_{nN+i}}{a_{nN+i}} \Big)^* - \frac{b_{(n+1)N+i}}{a_{(n+1)N+i}}
	\end{pmatrix},
\]
by \eqref{eq:56} we get
\[
	\| B_{(n+1)N+i}(z) - \overline{B_{nN+i}(z)} \| \leq
	\bigg| \frac{a_{(n+1)N+i-1}}{a_{(n+1)N+i}} - \bigg( \frac{a_{nN+i-1}}{a_{nN+i}} \bigg)^* \bigg| +
	|z| \bigg| \frac{\gamma}{a_{(n+1)N+i}} - \bigg( \frac{\gamma}{a_{nN+i}} \bigg)^* \bigg| +
	\bigg| \frac{b_{(n+1)N+i}}{a_{(n+1)N+i}} - \bigg( \frac{b_{nN+i}}{a_{nN+i}} \bigg)^* \bigg|.
\]
Thus, by the compactness of $K$ the result follows.
\end{proof}

\subsection{Asymptotically periodic case}
Let $N$ be a positive integer, and let $(\alpha_n : n \in \ZZ)$ and $(\beta_n : n \in \ZZ)$ be $N$-periodic complex sequences such that $\alpha_n \neq 0$ for any $n$. Let us define
\begin{equation} \label{eq:52}
	\frakX_i(x) =
	\prod_{j=i}^{N+i-1}
	\frakB_{j}(x)
	\quad \text{where} \quad
	\frakB_j(x) = 
	\begin{pmatrix}
		0 & 1 \\
		-\frac{\alpha_{j-1}}{\alpha_j} & \frac{x - \beta_j}{\alpha_j}
	\end{pmatrix}.
\end{equation}
Let $A_\per$ be the Jacobi matrix on $\ell^2(\NN_0)$ associated with the sequences 
$(\alpha_n : n \in \NN_0)$ and $(\beta_n : n \in \NN_0)$. Then $A_\per$ is bounded and
\[
	\sigmaEss(A_\per) = (\tr \frakX_0)^{-1} \big( [-2,2] \big),
\]
(see, e.g. \cite[Theorem 1]{Vazquez2000}). For a detailed study of $A_\per$, see \cite{Beckermann1997}
and \cite{Papanicolaou2019}.
It is known that $\sigmaEss(A_\per)$, as the subset of $\CC$, has empty interior and $\CC \setminus \sigmaEss(A_\per)$ is 
connected (see, e.g. \cite[Section 4.3]{Beckermann2001} or \cite[Lemma 2.5]{Beckermann1997}).

From the point of view of spectral theory it is natural to consider Jacobi matrices which are compact
perturbations of $A_\per$, that is
\[
	\lim_{n \to \infty} |a_n - \alpha_n| = 0, \quad
	\lim_{n \to \infty} |b_n - \beta_n| = 0.
\]
Let us consider the case when $\alpha_n > 0$ and $\beta_n \in \RR$. In \cite{Egorova2005}
were formulated conditions for $N=2$ assuring that the discrete spectrum of $A$ is empty. Finally, let us remark that the case of symmetric $\calA$ is well-developed, see e.g. \cite[Section 7.1]{SwiderskiTrojan2019} and the references therein.

We are ready to state our result.
\begin{corollary} \label{cor:2}
Let $N$ be a positive integer and $(\alpha_n : n \in \ZZ)$ and $(\beta_n : n \in \ZZ)$
be \emph{real} $N$-periodic sequences such that $\alpha_n > 0$ for every $n$. Suppose that the  sequences
$(a_n : n \in \NN_0)$ and $(b_n : n \in \NN_0)$ belong to $\twisD^N(\CC)$ and satisfy
\[
	\lim_{n \to \infty} |a_n - \alpha_n| = 0, \quad
	\lim_{n \to \infty} |b_n - \beta_n| = 0.
\]
Let $\frakX_0$ be defined by \eqref{eq:52} and let $K$ be a compact subset of
\[
	\{ x \in \RR : |\tr \frakX_0(x)| < 2 \}.
\]
Then there is a constant $c>1$ such that for every generalised eigenvector $u$ associated with $x \in K$
and for any $n \geq 1$
\[
	c^{-1} \big( |u_0|^2 + |u_1|^2 \big) \leq
	|u_{n-1}|^2 + |u_{n}|^2 \leq
	c \big( |u_0|^2 + |u_1|^2 \big).
\]
In particular, $\calA$ is proper, $K \cap \sigmaP(A) = \emptyset$ and 
$K \subset \sigma(A)$.
\end{corollary}
\begin{proof}
We are going to show that the hypotheses of Theorem~\ref{thm:A} are satisfied for any fixed $i \in \{0, 1, \ldots, N-1 \}$.

Observe that
\[
	\lim_{n \to \infty} \frac{a_{nN+i-1}}{|a_{nN+i-1}|} = 
	\lim_{n \to \infty} \frac{\alpha_{i-1}}{|\alpha_{i-1}|} = 1.
\]
Thus $\gamma = 1$. Let $z \in \CC$. Then for any $j \in \{0, 1, \ldots, N-1 \}$
\[
	\lim_{n \to \infty} B_{nN+j}(z) = \frakB_j(z)
\]
Hence,
\[
	\calX_i(z) := \lim_{n \to \infty} X_{nN+i}(z) = \frakX_i(z).
\]
It implies
\[
	\lim_{n \to \infty} \frac{a_{nN+i-1}}{a_{(n+1)N+i-1}} =
	\lim_{n \to \infty} \det X_{nN+i}(z) = \det \frakX_i(z) = 1
\]
and \eqref{eq:57} is satisfied. Moreover, for any $x \in K$ 
\[
	\frakX_i(x) \in \GL(2, \RR).
\]
Since
\[
	\frakX_{i} = 
	\big( \frakB_{i-1} \ldots  \frakB_0 \big) 
	\frakX_0 
	\big( \frakB_{i-1} \ldots  \frakB_0 \big)^{-1}
\]
one has $\discr \frakX_i = \discr \frakX_0$, and consequently, for any $x \in K$
\[
	\discr \calX_i(x) = \big( \tr \frakX_0(x) \big)^2 - 4 < 0.
\]

It remains to prove \eqref{eq:58}. By \eqref{eq:8}
\[
	\Big( \frac{1}{a_n} : n \in \NN \Big) \in \twisD^N(\CC).
\]
Thus, by \eqref{eq:7}
\[
	\Big( \frac{a_{n-1}}{a_n} : n \in \NN \Big),
	\Big( \frac{b_n}{a_n} : n \in \NN \Big) \in \twisD^N(\CC).
\]
By Proposition~\ref{prop:7}
\[
	(B_n : n \in \NN) \in \twisD^N \big(K, \GL(2, \CC) \big).
\]
Hence, by \eqref{eq:7}
\[
	(X_n : n \in \NN) \in \twisD^N \big(K, \GL(2, \CC) \big)
\]
and \eqref{eq:58} is proven. So the hypotheses of Theorem~\ref{thm:A} are satisfied 
and the conclusion follows.
\end{proof}

\subsection{Periodic modulations}
Let $N$ be a positive integer, and let $(\alpha_n : n \in \ZZ)$ and 
$(\beta_n : n \in \ZZ)$ be $N$-periodic complex sequences such that $\alpha_n \neq 0$ for any $n$. 
Let $\frakX_i$ be defined in \eqref{eq:52}. If the sequences $(a_n : n \in \NN_0)$ and 
$(b_n : n \in \NN_0)$ satisfy
\begin{equation} \label{eq:53}
	\lim_{n \to \infty} |a_n| = \infty, \quad
	\lim_{n \to \infty} \bigg| \frac{a_{n-1}}{a_n} - \frac{\alpha_{n-1}}{\alpha_n} \bigg| = 0 \quad\text{and} \quad
	\lim_{n \to \infty} \bigg| \frac{b_n}{a_n} - \frac{\beta_n}{\alpha_n} \bigg| = 0,
\end{equation}
then $A$ will be called a Jacobi matrix with periodically modulated entries. The case when $A$ is a 
symmetric operator is well-developed, see e.g. \cite[Section 7.2]{SwiderskiTrojan2019} and the 
references therein. It seems that there are virtually no results when $A$ is not symmetric. A
notable exception comes from the article \cite{Siegl2017}. Below we present a result in this direction.

\begin{corollary} \label{cor:3}
Let $N$ be a positive integer and let $(\alpha_n : n \in \ZZ)$ and $(\beta_n : n \in \ZZ)$ be 
\emph{real} $N$-periodic sequences such that $\alpha_n > 0$ for every $n$. Suppose that for some $\gamma \in \sS^1$
\[
	\lim_{n \to \infty} \frac{a_n}{|a_n|} = \gamma
\]
and
\begin{equation} \label{eq:55}
	\bigg( \frac{a_{n-1}}{a_n} : n \in \NN \bigg), \
	\bigg( \frac{b_n}{a_n} : n \in \NN \bigg), \
	\bigg( \frac{\gamma}{a_n} : n \in \NN \bigg) \in \twisD^N(\CC).
\end{equation}
Let \eqref{eq:53} be satisfied, $\frakX_0$ be defined by \eqref{eq:52} and $K$ be a compact subset 
of $\gamma \RR$. If $|\tr \frakX_0(0)| < 2$, then there is a constant $c>1$ such that for 
every generalised eigenvector $u$ associated with $z \in K$ and for any $n \geq 1$
\[
	c^{-1} \big( |u_0|^2 + |u_1|^2 \big) \leq
	|a_{n+N-1}| \big( |u_{n-1}|^2 + |u_{n}|^2 \big) \leq
	c \big( |u_0|^2 + |u_1|^2 \big).
\]
In particular, if
\[
	\sum_{n=0}^\infty \frac{1}{|a_n|} = \infty,
\]
then $\calA$ is proper, $\gamma \RR \cap \sigmaP(A) = \emptyset$ and $\gamma \RR \subset \sigma(A)$. 
Conversely, if
\[
	\sum_{n=0}^\infty \frac{1}{|a_n|} < \infty,
\]
then $\calA$ is improper, $\sigmaEss(A) = \emptyset$, $\sigma(A) = \CC$ and $\sigmaP(\Amax) = \CC$.
\end{corollary}
\begin{proof}
We are going to show that the hypotheses of Theorem~\ref{thm:A} are satisfied. Let us fix 
$i \in \{0, 1, \ldots, N-1 \}$.

Let $z \in \CC$. Then for any $j \in \{0, 1, \ldots, N-1 \}$
\[
	\lim_{n \to \infty} B_{nN+j}(z) = \frakB_j(0)
\]
Hence,
\[
	\calX_i(z) := \lim_{n \to \infty} X_{nN+i}(z) = \frakX_i(0).
\]
It implies
\[
	\lim_{n \to \infty} \frac{a_{nN+i-1}}{a_{(n+1)N+i-1}} =
	\lim_{n \to \infty} \det X_{nN+i}(z) = \det \frakX_i(0) = 1
\]
and \eqref{eq:57} is satisfied. Moreover,
\[
	\calX_i(z) = \frakX_i(0) \in \GL(2, \RR)
\]
and similarly as in the proof of Corollary~\ref{cor:2}
\[
	\discr \calX_i(x) = \big( \tr \frakX_0(0) \big)^2 - 4 < 0.
\]

It remains to prove \eqref{eq:58}. By Proposition~\ref{prop:7}
\[
	(B_n : n \in \NN) \in \twisD^N \big(K, \GL(2, \CC) \big).
\]
Hence, by \eqref{eq:7}
\[
	(X_n : n \in \NN) \in \twisD^N \big(K, \GL(2, \CC) \big)
\]
and \eqref{eq:58} is proven. So the hypotheses of Theorem~\ref{thm:A} are satisfied 
and the conclusion follows.
\end{proof}

The following proposition gives a simple method of construction sequences satisfying hypotheses of
Corollary~\ref{cor:3}.

\begin{proposition} \label{prop:9}
Let $N$ be a positive integer and $\gamma \in \sS^1$. Let $(\alpha_n : n \in \ZZ)$ and 
$(\beta_n : n \in \ZZ)$ be $N$-periodic sequences of positive and real numbers, respectively. Suppose
we are given a complex sequence $(\tilde{a}_n : n \in \NN_0)$ such that $\tilde{a}_n \neq 0$ for all $n$, and
\begin{equation} \label{eq:72}
	\lim_{n \to \infty} |\tilde{a}_n| = \infty, \qquad
	\lim_{n \to \infty} \frac{\tilde{a}_n}{|\tilde{a}_n|} = \gamma, \qquad
	\lim_{n \to \infty} \frac{\tilde{a}_{n-1}}{\tilde{a}_n} = 1
\end{equation}
and
\begin{equation} \label{eq:71}
	\bigg( \frac{\tilde{a}_{n-1}}{\tilde{a}_n} : n \in \NN \bigg) \in \twisD^N(\CC).
\end{equation}
Set
\[
	a_n = \alpha_n \tilde{a}_n, \qquad 
	b_n = \beta_n \tilde{a}_n.
\]
If $|\frakX_0(0)|<2$, where $\frakX_0$ is defined in \eqref{eq:52}, then the Jacobi matrix corresponding
to the sequences $(a_n : n \in \NN_0)$ and $(b_n : n \in \NN_0)$ satisfy the hypotheses of 
Corollary~\ref{cor:3}.
\end{proposition}
\begin{proof}
The condition \eqref{eq:55} follows from Proposition~\ref{prop:2} applied to \eqref{eq:71} and
\begin{equation} \label{eq:73}
	\frac{a_{n-1}}{a_n} = \frac{\alpha_{n-1}}{\alpha_n} \frac{\tilde{a}_{n-1}}{\tilde{a}_n}, \qquad
	\frac{b_n}{a_n} = \frac{\beta_n}{\alpha_n}, \qquad
	\frac{\gamma}{a_n} = \frac{1}{\alpha_n} \frac{\gamma}{\tilde{a}_n}.
\end{equation}
Finally, condition \eqref{eq:53} follows from \eqref{eq:72}, \eqref{eq:73} and
\[
	|a_n| = \alpha_n |\tilde{a}_n|, \qquad
	\frac{a_n}{|a_n|} = \frac{\tilde{a}_n}{|\tilde{a}_n|}.
\]
The proof is complete.
\end{proof}

\subsection{Additive perturbations}
\begin{proposition} \label{prop:8}
Suppose that the Jacobi matrix $A$ satisfies the hypotheses of Corollary~\ref{cor:3}. 
Let $(x_n : n \in \NN_0)$ and $(y_n : n \in \NN_0)$ be sequences such that
\begin{equation} \label{eq:60}
	\bigg( \frac{x_n}{a_n} : n \in \NN \bigg),
	\bigg( \frac{y_n}{a_n} : n \in \NN \bigg) \in \twisD^N(\CC)
\end{equation}
and
\begin{equation} \label{eq:54}
	\lim_{n \to \infty} \frac{x_n}{a_n} = 0, \qquad 
	\lim_{n \to \infty} \frac{y_n}{a_n} = 0.
\end{equation}
Define 
\[
	\tilde{a}_n = a_n + x_n, \qquad \tilde{b}_n = b_n + y_n.
\]
If $\tilde{a}_n \neq 0$ for all $n$, then the Jacobi matrix corresponding to the sequences 
$(\tilde{a}_n : n \geq 0)$ and $(\tilde{b}_n : n \geq 0)$ satisfy the hypotheses of Corollary~\ref{cor:3}
for the same $\alpha, \beta$ and $\gamma$.
\end{proposition}
\begin{proof}
We have
\[
	\frac{\tilde{a}_n}{a_n} = 1 + \frac{x_n}{a_n}, \qquad 
	\frac{\tilde{b}_n}{a_n} = \frac{b_n}{a_n} + \frac{y_n}{a_n}.
\]
Hence, by \eqref{eq:60} and Proposition~\ref{prop:1}
\[
	\bigg( \frac{\tilde{a}_n}{a_n} : n \in \NN \bigg), 
	\bigg( \frac{\tilde{b}_n}{a_n} : n \in \NN \bigg) \in \twisD^N(\CC)
\]
Moreover, by \eqref{eq:54}
\begin{equation} \label{eq:61}
	\lim_{n \to \infty} \frac{\tilde{a}_n}{a_n} = 1, \qquad
	\lim_{n \to \infty} \bigg| \frac{\tilde{b}_n}{a_n} - \frac{\beta_n}{\alpha_n} \bigg| = 0.
\end{equation}
Therefore, by Proposition~\ref{prop:2}
\[
	\bigg( \frac{a_n}{\tilde{a}_n} : n \in \NN \bigg) \in \twisD^N(\CC).
\] 
Thus, the condition \eqref{eq:55} is implied by
\begin{equation} \label{eq:43}
	\frac{\tilde{a}_{n-1}}{\tilde{a}_{n}} = 
	\frac{\tilde{a}_{n-1}}{a_{n-1}} \frac{a_{n-1}}{a_n} \frac{a_n}{\tilde{a}_{n}}, \qquad
	\frac{\tilde{b}_n}{\tilde{a}_n} = 
	\frac{a_n}{\tilde{a_n}} \frac{\tilde{b}_n}{a_n}, \qquad
	\frac{\gamma}{\tilde{a}_n} = 
	\frac{a_n}{\tilde{a}_n} \frac{\gamma}{a_n}
\end{equation}
and Proposition~\ref{prop:2}. Finally, condition \eqref{eq:53} follows from \eqref{eq:61}, \eqref{eq:43}
and 
\[
	|\tilde{a}_n| = |a_n| \bigg| \frac{\tilde{a}_n}{a_n} \bigg|, \qquad 
	\frac{\tilde{a}_n}{|\tilde{a}_n|} = 
	\frac{\tilde{a}_n}{a_n} \frac{a_n}{|a_n|} \bigg| \frac{a_n}{\tilde{a}_n} \bigg|.
\]
The proof is complete.
\end{proof}

The following corollary concerns complex perturbations of real Jacobi matrices.
\begin{corollary} \label{cor:4}
Suppose that $(a_n : n \in \NN_0)$ and $(b_n : n \in \NN_0)$ satisfy the hypotheses of 
Corollary~\ref{cor:3}. Suppose that $a_n > 0$ and $b_n \in \RR$ for every $n$. Moreover, let the 
real sequences $(x_n : n \geq 0)$ and $(y_n : n \geq 0)$ are such that
\[
	\sum_{n=0}^\infty \bigg| \frac{x_{n+N}}{a_{n+N}} - \frac{x_{n}}{a_{n}} \bigg| +
	\sum_{n=0}^\infty \bigg| \frac{y_{n+N}}{a_{n+N}} - \frac{y_{n}}{a_{n}} \bigg| < \infty
\]
and
\[
	\lim_{n \to \infty} \frac{x_n}{a_n} = 0, \quad
	\lim_{n \to \infty} \frac{y_n}{a_n} = 0.
\]
Define 
\[
	\tilde{a}_n = a_n + i \epsilon_n x_n, \quad
	\tilde{b}_n = b_n + i \epsilon_n y_n,
\]
where $\epsilon_n = (-1)^{\lfloor n/N \rfloor}$. Then the sequences $(\tilde{a}_n : n \geq 0)$ and 
$(\tilde{b}_n : n \geq 0)$ again satisfy the assumptions of Corollary~\ref{cor:3} for the same 
$\alpha, \beta$ and~$\gamma$.
\end{corollary}
\begin{proof}
Let $j \in \{ 0, 1, \ldots, N-1 \}$. Since the sequences $a$ and $x$ are real valued, one has
\[
	\twisV \bigg( \frac{i \epsilon_{nN+j} x_{nN+j}}{a_{nN+j}} : n \in \NN \bigg) = 
	\sum_{n=1}^\infty 
	\bigg|i \epsilon_{(n+1)N+j} 
	\bigg( \frac{x_{(n+1)N+j}}{a_{(n+1)N+j}} - \frac{x_{nN+j}}{a_{nN+j}} \bigg) \bigg| =
	\sum_{n=1}^\infty \bigg| \frac{x_{(n+1)N+j}}{a_{(n+1)N+j}} - \frac{x_{nN+j}}{a_{nN+j}} \bigg| < \infty
\]
and similarly
\[
	\twisV \bigg( \frac{i \epsilon_{nN+j} y_{nN+j}}{a_{nN+j}} : n \in \NN \bigg) < \infty.
\]
Hence, the conclusion follows from Proposition~\ref{prop:8}.
\end{proof}

Let us illustrate Corollary~\ref{cor:4} with the following example.
\begin{example}
Let $N$ be a positive integer and let $(\alpha_n : n \in \ZZ)$ and $(\beta_n : n \in \ZZ)$ be
$N$-periodic sequences of positive and real numbers, respectively. Suppose that $|\tr \frakX_0(0)| < 2$,
where $\frakX_0$ is defined in \eqref{eq:52}. Let $0 \leq \mu < \lambda$. 
Define
\[
	a_n = \alpha_n (n+1)^\lambda + i (-1)^{\lfloor n/N \rfloor} (n+1)^\mu, \quad
	b_n = \beta_n (n+1)^\lambda + i (-1)^{\lfloor n/N \rfloor} (n+1)^\mu.
\]
Then for any compact $K \subset \RR$ there is a constant $c>1$ such that for 
every generalised eigenvector $u$ associated with $z \in K$ and for any $n \geq 1$
\[
	c^{-1} \big( |u_0|^2 + |u_1|^2 \big) \leq
	|a_{n+N-1}| \big( |u_{n-1}|^2 + |u_{n}|^2 \big) \leq
	c \big( |u_0|^2 + |u_1|^2 \big).
\]
To prove this claim let us observe that the Proposition~\ref{prop:9} applied to
\[
	\tilde{a}_n = (n+1)^\lambda
\]
implies that the sequences
\[
	\hat{a}_n = \alpha_n (n+1)^\lambda, \qquad
	\hat{b}_n = \beta_n (n+1)^\lambda
\]
satisfy the assumptions of Corollary~\ref{cor:3}. Let
\[
	x_n = y_n = (n+1)^\mu.
\]
Since
\[
	\frac{x_n}{\hat{a}_n} = \frac{1}{\alpha_n} \frac{1}{(n+1)^{\lambda-\mu}}
\]
it satisfies conditions \eqref{eq:60} and \eqref{eq:54}. Therefore, the conclusion follows from 
Corollary~\ref{cor:4}.
\end{example}

\subsection{Blend}
Let $N$ be a positive integer, and let $(\alpha_n : n \in \ZZ)$ and $(\beta_n : n \in \ZZ)$ be 
$N$-periodic sequences of complex numbers such that $\alpha_n \neq 0$ for any $n$. Suppose that 
sequences $\tilde{a}, \tilde{b}, \tilde{c}$ and $\tilde{d}$ satisfy
\begin{equation} \label{eq:62}
	\lim_{n \to \infty} \big|\tilde{a}_n - \alpha_n\big| = 0, \qquad 
	\lim_{n \to \infty} \big|\tilde{b}_n - \beta_n\big| = 0, \qquad
	\lim_{n \to \infty} |\tilde{c}_n| = \infty
\end{equation}
and for every $n$ one has $\tilde{a}_n \neq 0$ and $\tilde{c}_n \neq 0$. 
For $k \geq 0$ and $i \in \{0, 1, \ldots, N+1 \}$, we define
\begin{equation} \label{eq:63}
		a_{k(N+2)+i} = 
		\begin{cases}
			\tilde{a}_{kN+i} & \text{if } i \in \{0, 1, \ldots, N-1\}, \\
			\tilde{c}_{2k} & \text{if } i = N, \\
			\tilde{c}_{2k+1} & \text{if } i = N+1,
		\end{cases}, \qquad
		b_{k(N+2)+i} = 
		\begin{cases}
			\tilde{b}_{kN+i} & \text{if } i \in \{0, 1, \ldots, N-1\}, \\
			\tilde{d}_{2k} & \text{if } i = N, \\
			\tilde{d}_{2k+1} & \text{if } i = N+1
		\end{cases}.
\end{equation}

Jacobi matrix $A$ will be called \emph{$N$-periodic blend} if it satisfies \eqref{eq:62} and 
\eqref{eq:63}. This class of matrices has been studied previously in \cite[Theorem 5]{Janas2011},
\cite[Section 7.3]{SwiderskiTrojan2019} and \cite{ChristoffelI} but only in the case when $A$ is self-adjoint.

\begin{corollary}
Let Jacobi matrix $A$ be $N$-periodic blend, where $\alpha_n > 0$ and $\beta_n \in \RR$ for any $n$. 
Suppose that
\begin{equation} \label{eq:64}
	\lim_{n \to \infty} \frac{\tilde{c}_{2n}}{\tilde{c}_{2n+1}} = 1, \qquad
	\lim_{n \to \infty} \frac{\tilde{c}_n}{|\tilde{c}_n|} = 1, \qquad
	\lim_{n \to \infty} \tilde{d}_{2n} = \delta,
\end{equation}
and
\begin{equation} \label{eq:65}
	\big( 1/a_n : n \in \NN \big), 
	\big( b_n : n \in \NN \big) \in \twisD^{N+2}(\CC) \quad \text{and} \quad
	\bigg( \frac{a_{n(N+2)+N}}{a_{n(N+2)+N+1}} : n \in \NN \bigg) \in \twisD(\CC).
\end{equation}
For $i \in \{1, 2, \ldots, N \}$ we set
\[
	\calX_i(z) = 
	\bigg\{ \prod_{j=1}^{i-1} \frakB_j(z) \bigg\} 
	\calC(z) 
	\bigg\{ \prod_{j=i}^{N-1} \frakB_j(z) \bigg\},
\] 
where $\frakB_j$ is defined in \eqref{eq:52} and
\[
	\calC(z) =
	\begin{pmatrix}
		0 & -1 \\
		-\frac{\alpha_{N-1}}{\alpha_0} & -\frac{2z - \beta_0 - \delta}{\alpha_0}
	\end{pmatrix}.
\] 
Let $K$ be a compact subset of
\[
	\big\{ x \in \RR : | \tr \calX_1(x) | < 2 \big\}.
\]
Then there is a constant $c>1$ such that for any generalised eigenvector associated with $z \in K$, 
any $n \geq 1$ and any $i \in \{1, 2, \ldots, N \}$
\begin{equation} \label{eq:70}
	c^{-1} \big( |u_0|^2 + |u_1|^2 \big) \leq
	|u_{n(N+2)+i-1}|^2 + |u_{n(N+2)+i}|^2 \leq
	c \big( |u_0|^2 + |u_1|^2 \big).	
\end{equation}
In particular, $\calA$ is proper, $K \cap \sigmaP(A) = \emptyset$ and $K \subset \sigma(A)$.
\end{corollary}
\begin{proof}
We are going to show that the hypotheses of Theorem~\ref{thm:A} are satisfied for any fixed $i \in \{1, 2, \ldots, N \}$.

By \eqref{eq:62} and \eqref{eq:63}
\begin{equation} \label{eq:69}
	\lim_{n \to \infty} \frac{a_{n(N+2)+i-1}}{|a_{n(N+2)+i-1}|} = 
	\frac{\alpha_{i-1}}{|\alpha_{i-1}|} = 1.
\end{equation}
Thus $\gamma = 1$. Define
\[
	\tilde{B}_n(z) =
	\begin{pmatrix}
		0 & 1 \\
		-\frac{\tilde{a}_{n-1}}{\tilde{a}_{n}} & \frac{z - \tilde{b}_n}{\tilde{a}_n}
	\end{pmatrix}.
\]
For any $j \in \{ 1, 2, \ldots, N-1 \}$ one has $B_{n(N+2)+j} = \tilde{B}_{nN+j}$. Thus,
\begin{equation} \label{eq:148}
	X_{n(N+2)+i} = 
	\prod_{j=i}^{N+2+i-1} B_{n(N+2)+j} = 
	\Bigg\{
	\prod_{j=1}^{i-1} \tilde{B}_{(n+1)N+j}
	\Bigg\}
	C_n
	\Bigg\{
	\prod_{j=i}^{N-1} \tilde{B}_{nN+j}
	\Bigg\},
\end{equation}
where
\[
	C_n = B_{(n+1)(N+2)} B_{n(N+2)+N+1} B_{n(N+2)+N}.
\]
A direct computation shows that
\begin{align} \label{eq:67}
	C_n(z) &= 
	\begin{pmatrix}
		0 & -\frac{a_{n(N+2)+N}}{a_{n(N+2)+N+1}} \\
		\frac{a_{n(N+2)+N-1}}{a_{(n+1)(N+2)}} \frac{a_{n(N+2)+N+1}}{a_{n(N+2)+N}} & 
		-\frac{z - b_{(n+1)(N+2)}}{a_{(n+1)(N+2)}} \frac{a_{n(N+2)+N}}{a_{n(N+2)+N+1}}
		-\frac{z - b_{n(N+2)+N}}{a_{(n+1)(N+2)}} \frac{a_{n(N+2)+N+1}}{a_{n(N+2)+N}}
	\end{pmatrix} \\ \nonumber
	&+ \frac{z - b_{n(N+2)+N+1}}{a_{n(N+2)+N+1}}
	\begin{pmatrix}
		-\frac{a_{n(N+2)+N-1}}{a_{n(N+2)+N}} & \frac{z - b_{n(N+2)+N}}{a_{n(N+2)+N}} \\
		-\frac{a_{n(N+2)+N-1}}{a_{n(N+2)+N}} \frac{z - b_{(n+1)(N+2)}}{a_{(n+1)(N+2)}} &
		\frac{z - b_{n(N+2)+N}}{a_{n(N+2)+N}} \frac{z - b_{(n+1)(N+2)}}{a_{(n+1)(N+2)}}	
	\end{pmatrix}.
\end{align}
In particular,
\[
	\lim_{n \to \infty} C_n(z) = \calC(z),
\]
and since for any $j \in \{1, 2, \ldots, N-1 \}$
\[
	\lim_{n \to \infty} \tilde{B}_{nN+j}(z) = \frakB_j(z),
\]
one obtains
\begin{equation} \label{eq:68}
	\lim_{n \to \infty} X_{n(N+2)+i}(z) = \calX_i(z).
\end{equation}
By \eqref{eq:62} and \eqref{eq:63} one has
\[
	\lim_{n \to \infty} \frac{a_{n(N+2)+i-1}}{a_{(n+1)(N+2)+i-1}} = \frac{\alpha_{i-1}}{\alpha_{i-1}} = 1.
\]
Hence, by \eqref{eq:69} and \eqref{eq:68} the condition \eqref{eq:57} is satisfied.

Observe that by Proposition~\ref{prop:1} combined with \eqref{eq:63}, \eqref{eq:64} and \eqref{eq:65} one has $\delta \in \RR$. Thus, for any $x \in K$
\[
	\calX_i(x) \in \GL(2, \RR).
\]
Since
\[
	\calX_i (x)= 
	\bigg\{ \prod_{j=1}^{i-1} \frakB_j(x) \bigg\} 
	\calX_1(x)
	\bigg\{ \prod_{j=1}^{i-1} \frakB_j(x) \bigg\}^{-1}
\]
one obtains
\[
	\discr \calX_i(x) = 
	\discr \calX_1(x) = 
	\big( \tr \calX_1(x) \big)^2 - 4 < 0,
\]
and consequently, $K \subset \Lambda$.

Proposition~\ref{prop:2} with \eqref{eq:65} implies
\begin{equation} \label{eq:147a}
	\bigg( \frac{b_n}{a_n} : n \in \NN \bigg) \in \twisD^{N+2}(\CC)
\end{equation}
and for every $j \in \{0, 1, \ldots, N-1 \}$
\begin{equation} \label{eq:147c}
	\big( a_{n(N+2)+j} : n \in \NN \big) \in \twisD(\CC).
\end{equation}
Thus, by Proposition~\ref{prop:2} with \eqref{eq:65} and \eqref{eq:147c},
one obtains that for each $j' \in \{1, \ldots, N+1 \}$
\begin{equation} \label{eq:147b}
	\bigg( \frac{a_{n(N+2)+j'-1}}{a_{n(N+2)+j'}} : n \in \NN \bigg) \in \twisD(\CC).
\end{equation}
By combining \eqref{eq:65}, \eqref{eq:147a} and \eqref{eq:147b} with Proposition~\ref{prop:7} one gets for any $j \in \{0, 1, \ldots, N-1 \}$
\begin{equation} \label{eq:147}
	(\tilde{B}_{nN+j} : n \in \NN) \in \twisD \big( K, \GL(2, \CC) \big).
\end{equation}
Finally, 
and by repeated application of Propositions \ref{prop:1} and \ref{prop:2} to \eqref{eq:67} we can verify that
\[
	\big( C_{n} : n \in \NN \big) \in \twisD \big(K, \GL(2,\CC) \big),
\]
which together with \eqref{eq:148} and \eqref{eq:147} implies
\[
	\big( X_{n(N+2)+i} : n \in \NN \big) \in \twisD \big(K, \GL(2,\CC) \big).
\]
Consequently, the condition \eqref{eq:58} is satisfied. By \eqref{eq:147c} the sequence $\big( a_{n(N+2)+i-1} : n \in \NN \big)$ is bounded. 
Therefore, \eqref{eq:70} follows from Theorem~\ref{thm:A}. Finally, by \eqref{eq:70} we obtain that $u$ is not square summable
and by Proposition~\ref{prop:4} the result follows.
\end{proof}
\begin{remark}
For $i \in \{0, N+1 \}$ the hypotheses of Theorem~\ref{thm:A} are \emph{not} satisfied so we cannot claim that \eqref{eq:70} holds also
in this case. In fact, for $i=0$ the bound \eqref{eq:70} holds but it is not the case for $i=N+1$ (see the proof of 
\cite[Claim 4.12]{ChristoffelI}).
\end{remark}

\begin{bibliography}{jacobi}
\bibliographystyle{amsplain}
\end{bibliography}

\end{document}